\documentclass{amsart}

\usepackage{amsfonts}
\usepackage{amsmath}
\usepackage{amsthm}
\usepackage{hyperref}
\usepackage{color}
\usepackage{tikz-cd}
\usepackage{mathtools}
\usepackage{graphicx}
\usepackage{bm}
\usepackage[letterpaper,asymmetric,left=1.1in,right=1.1in,top=1.0in,bottom=1.0in,bindingoffset=0.0in]{geometry}

\newtheorem{proposition}{Proposition}
\numberwithin{equation}{section}

\newtheorem*{symmetry theorem}{Symmetry Theorem}

\DeclareFontFamily{OT1}{pzc}{}
\DeclareFontShape{OT1}{pzc}{m}{it}{<-> s * [1.200] pzcmi7t}{}
\DeclareMathAlphabet{\mathpzc}{OT1}{pzc}{m}{it}

\begin{document}

\title{Symmetries of the Three Gap Theorem}
\author{Aneesh Dasgupta, Roland Roeder}
\maketitle

\begin{abstract}
    The Three Gap Theorem states that for any $\alpha \in \mathbb{R}$ and $N \in \mathbb{N}$, the fractional parts of $\{ 0\alpha, 1\alpha, \dots, (N - 1)\alpha \}$ partition the unit circle into gaps of at most three distinct lengths. We prove a result about symmetries in the order with which the sizes of gaps appear on the circle.
\end{abstract}

Choose an irrational angle $\alpha$ measured in ``turns,'' where one turn corresponds to $2\pi$ radians, and plot the points on the circle at angles 
        $$ 0, \quad \alpha, \quad 2 \alpha, \quad 3 \alpha, \dots, (N - 1) \alpha . $$
    For $\alpha = \sqrt{2}$ and $N = 27$, one obtains Figure 1.
    
    \begin{figure}[h!]
        \begin{center}
            \scalebox{0.8}{
            \includegraphics[scale=0.43]{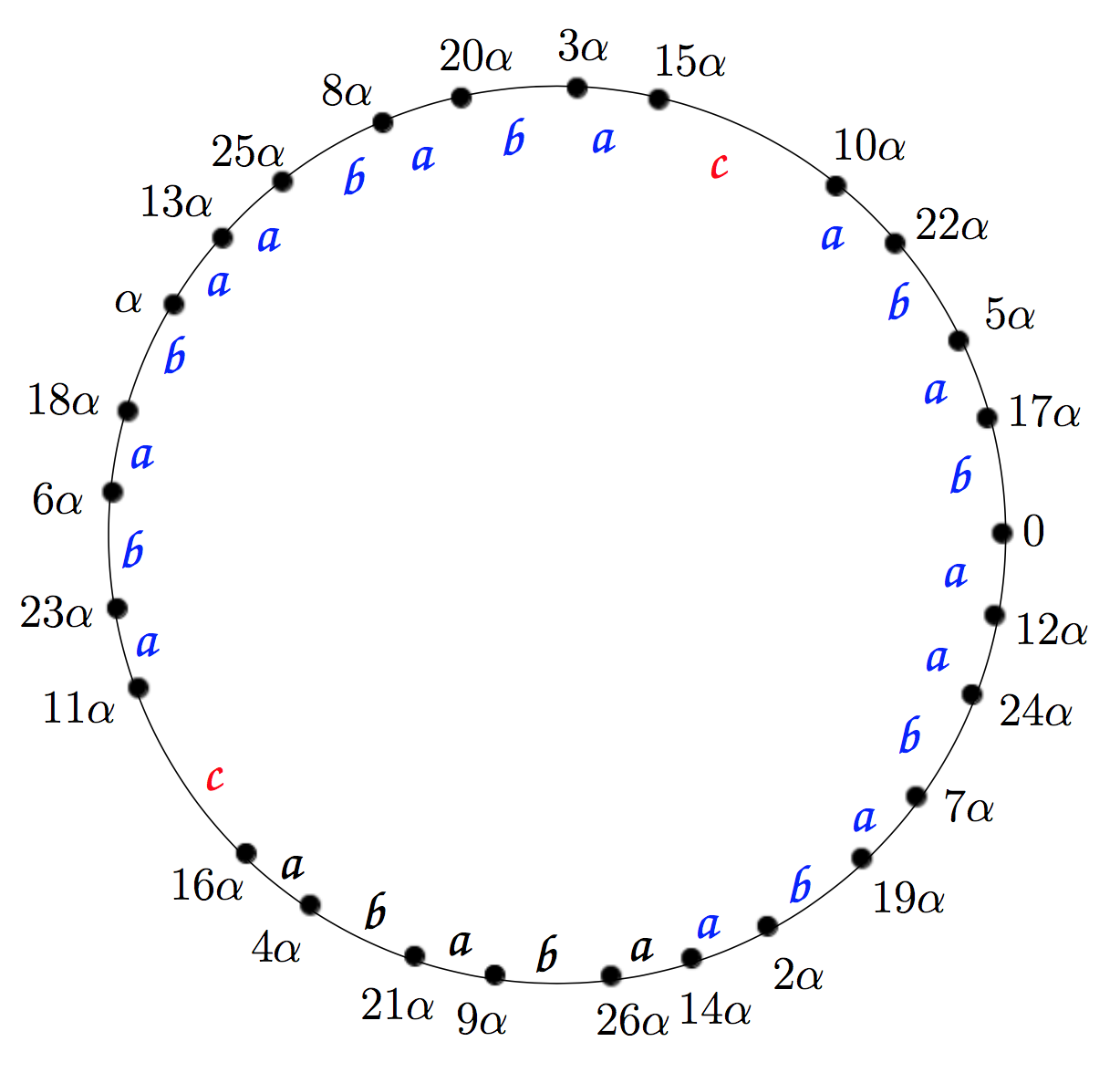}
        }
        \end{center}
        \caption{Illustration of the $3$-gap theorem for $\alpha = \sqrt{2}$ and $N=20$.  The short gaps
        are labeled with $\mathpzc{a}$, the medium gaps with $\mathpzc{b}$, and the longest gaps with $\mathpzc{c}$.
        One example of the reflectional symmetry that is proved in the Symmetry Theorem is indicated with the 
        center of symmetry being the red $\mathpzc{c}$ between $10\alpha$ and $15\alpha$, and the symmetric letters shown in blue. The length of the symmetry is indicated by the next
        closest $\mathpzc{c}$ gap, which is also shown in red (between $11\alpha$ and $16 \alpha$).}
    \end{figure}

A surprising observation is that for any choice of $N$ and $\alpha$, the distances (gaps) between consecutive points on the circle attain only three values. This is the content of the famous Three Gap Theorem, proved by S\'os, Sur\'anyi, and \'Swierczkowski in the 1950s, and it can be seen in the special case of Figure 1. In this paper we present a curious symmetry in how the sizes of the gaps are distributed on the circle. See the Symmetry Theorem and Figure 3 below.

% \vspace{0.15in}
%     \textbf{Setup }
% \section{Introduction}
\subsection{Setup}
In order to work more carefully, it is convenient to represent the circle as the interval $[0, 1]$ with the endpoints identified. We will now rephrase the setup in this context, and state the Three Gap Theorem more precisely.

    Let $\alpha \in \mathbb{R} \setminus \mathbb{Q}$ and $N \in \mathbb{N}$, and for any real number $x$, denote the fractional part as $\{ x \}$. We order the numbers $\{m\alpha \}$, where $0 \leq m < N$, into the sequence
    	$$0 = y_0(N) < y_1(N) < \dots < y_{N - 1}(N) < 1.$$
    We then consider the differences between consecutive numbers in the sequence, called $\textit{gaps}$ (or \textit{spacings}),
    	$$\delta_j(N) = y_{j + 1}(N) - y_j(N),$$
    for $j = 0, \dots, N - 2$ and $\delta_{N - 1}(N) = 1 - y_{N - 1}(N)$. Now, let $D(N)$\ be the number of distinct gaps and let $\Delta_{j}(N)$\ be the ordered sequence of distinct gaps from the $\delta_j(N)$, so that 
    	$$0 < \Delta_1(N) < \dots < \Delta_{D(N)}(N) < 1.$$
    	
    \begin{figure}[h!]
        \begin{center}
            \includegraphics[scale=0.43]{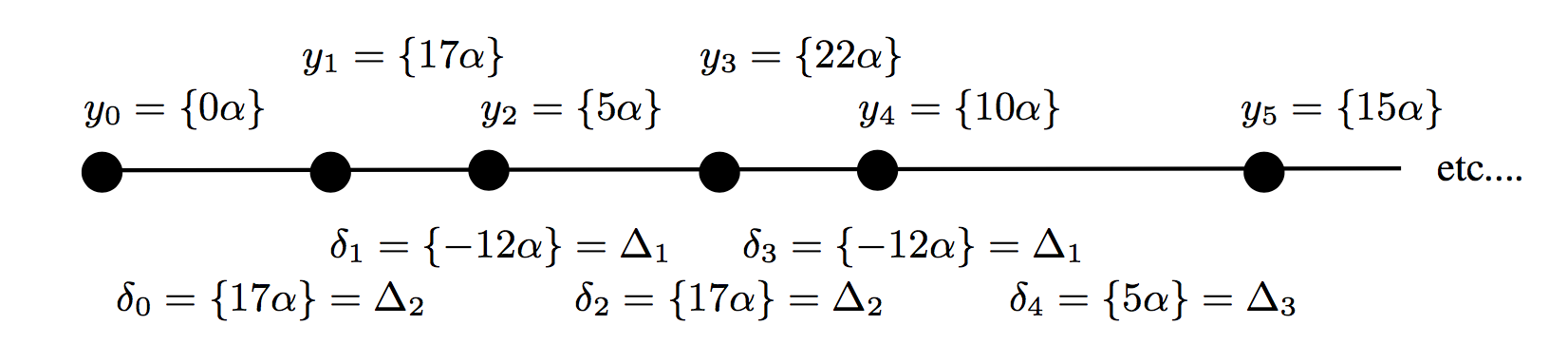}
        \end{center}
        \caption{Illustration of the definitions of $y_j(N)$, $\delta_j(N)$, and $\Delta_k(N)$ in the case of $\alpha = \sqrt{2}$ and $N = 27$ (these are the same values as in Figure 1).  Only the 
        left part of the interval is shown, and the dependence on $N$ is dropped from the notation.}
    \end{figure}        	    	
    	
    \vspace{0.1in}
    % \noindent
    % \begin{namedthm}{Three Gap Theorem}[\cite{THREE_GAP_1}, \cite{THREE_GAP_2}, \cite{THREE_GAP_3}]
        % For any choice of $\alpha, N$ as above, $D(N) \leq 3$.
    % \end{namedthm}
    
    \noindent
    \textbf{Three Gap Theorem} (S\'os, Sur\'anyi, and \'Swierczkowski). \emph{$D(N) \leq 3$ for any choice of $\alpha$ and $N$}.
    
    \vspace{.15in}
    For the original references, see \cite{THREE_GAP_1, THREE_GAP_2, THREE_GAP_3}. Since the original proofs, there have been many new proofs and interpretations of the Three Gap Theorem. For example, see \cite{LIANG, MARKLOF_1, RAVENSTEIN}. Note also the various higher-dimensional versions of the problem that have been recently discussed, see \cite{DYSON, HIGHSCHOOL, MARKLOF_2, MARKLOF_3}.  

    % \vspace{.15in}
    
    \noindent
    % Throughout the paper we use RHS to denote the right hand side of an equation and LHS to denote the left hand side.

\subsection{Words in the Gap Lengths}
    We will describe the order with which the sizes of gaps occur on the circle with a word $W$ in the letters $\mathpzc{a}, \mathpzc{b}, \text{ and } \mathpzc{c}$. More specifically, we define the $j^{\text{th}}$ letter $W_j$, of the word $W$, to be $\mathpzc{a}, \mathpzc{b} \text{ or } \mathpzc{c}$ corresponding to the gap $\delta_j(N)$, with $\mathpzc{a}$ corresponding to the smallest gap, $\mathpzc{b}$ the medium-sized gap, and $\mathpzc{c}$ the largest gap. We interpret the word cyclically so that $W_j = W_{j \mod N}$. When it is necessary to indicate the dependence on $N$, we will denote the word $W$ as $W(N)$.

    \begin{symmetry theorem} \label{SYM}
        Fix any $\alpha \in \mathbb{R} \setminus \mathbb{Q}$ and $N \in \mathbb{N}$, and let $W$ be the word generated by the corresponding gaps on the circle. Then for any $\mathpzc{c}$ in $W$, the $k^{\text{th}}$ letter to the right of it is always the same as the $k^{\text{th}}$ letter to the left of it, so long as the index $k$ is smaller than the index of the first $\mathpzc{c}$ occurrence on either side. 
        
        More precisely, if $W_{J} = \mathpzc{c}$, then $W_{J - k} = W_{J + k}$ for $k = 0, \dots, \ell$ where $\ell + 1$ is the smallest index such that $W_{J - (\ell + 1)} = \mathpzc{c} \text{ or } W_{J + (\ell + 1)} = \mathpzc{c}$.
    \end{symmetry theorem}
        
    \begin{figure}[h!]
        \begin{center}
            \includegraphics[scale=0.43]{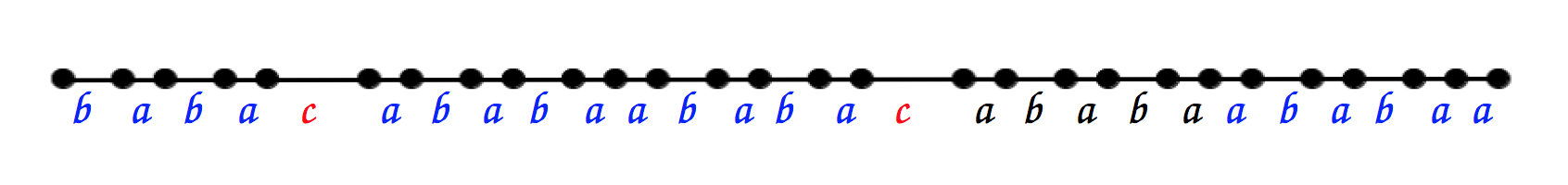}
            %\scalebox{1.1}{
            %\input{segment2.pstex_t}
            %}
        \end{center}
        \caption{Illustration of the Symmetry Theorem in the case of $\alpha = \sqrt{2}$ and $N = 27$ (the same values as for Figures 1 and 2). The word in $\mathpzc{a}, \mathpzc{b}, \text{ and } \mathpzc{c}$ is symmetric about either $\mathpzc{c}$. In the case of the leftmost one, the part of the word that is symmetric is shown in blue.
        }
    \end{figure}
    
    The symmetry becomes more impressive for larger values of $N$.   For example,
    the following is the word generated when $\alpha = \sqrt{2}$ and $N = 67$, and the ``limiting" $\mathpzc{c}$ happens to occur on both sides:
        $$ { \color{blue} \mathpzc{aababaababa} } { \color{red}\mathpzc{c} } { \color{blue}\mathpzc{ababaababaabababaababaababa} } { \color{red} \mathpzc{c} } \mathpzc{ababaababa} { \color{red} \mathpzc{c} } { \color{blue} \mathpzc{ababaababaababab} } $$
        
\section{Proof of the Symmetry Theorem}

    Our proof closely follows the ideas and notations of van Ravenstein \cite{RAVENSTEIN}. Let $u_j$ be the ordered sequence such that $\{ u_j \alpha \} < \{ u_{j + 1} \alpha \}$, that is, the order each point appears on the circle when identified by the number realizing it. Specifically, $(u_0, \dots, u_{N - 1})$ is a permutation of $(0, 1, 2, \dots, N - 1)$ and $y_j(N) = \{u_j \alpha \}$. Note that we are deviating from the conventions of \cite{RAVENSTEIN} who uses $(u_1, \dots, u_N)$ instead. We interpret the $u_j$ cyclically so that $u_j = u_{j \mod N}$.
    
    Remark that at certain choices of $N$ there will be only two sizes of gaps corresponding to the symbols $\mathpzc{a}$ and $\mathpzc{b}$. 
    When one more point is added, it will result in gaps of a new size, and the labelings will have to be updated. Therefore, we will call such times $N$ where there are only two sizes of gaps the \emph{relabeling times}, and denote the ordered sequence of relabeling times as $(R_k)_{k \geq 1}$.
        
    We will rely on three basic facts from van Ravenstein \cite{RAVENSTEIN}:
    
    \noindent \textbf{Fact 1:} $N$ is a relabeling time if and only if $N = u_1 + u_{N - 1}$. Indeed, in this case, adding the point $\{ N\alpha \}$ results in a point closer to $0$ than either $\{ u_1 \alpha \}$ or $\{ u_{N - 1} \alpha \}$, and thus gaps of a new size appear. Note that $u_1$ and $u_{N - 1}$ depend upon $N$, and this is a somewhat subtle condition which can be expressed nicely in terms of the continued fraction of $\alpha$, but it is not needed in our paper. 
    
    \vspace{0.1in}
    \noindent \textbf{Fact 2:} At a relabeling time $N$, we have:
        \begin{align} u_j = j \cdot u_1 \mod N \qquad j = 0, 1, \dots, N - 1 . \label{MODD} \end{align}
    
    \vspace{0.1in}
    \noindent \textbf{Fact 3:} If $N$ is not a relabeling time, then adding the point $\{ N \alpha \}$ results in splitting a gap labeled by $\mathpzc{c}$ into a gap labeled by $\mathpzc{a}$ and a gap labeled by $\mathpzc{b}$ in either possible order.
    
    \vspace{0.1in}
    \noindent Facts 1 and 2 are found in \cite{RAVENSTEIN}[Lemma 2.1] and Fact 3 is found in \cite{RAVENSTEIN}[Theorem 2.2]. 
    
    \begin{proposition} \label{PROP1}
        Let $W$ be the word of $\mathpzc{a}$'s and $\mathpzc{b}$'s at a relabeling time. Then the word satisfies the following symmetry. 
        Let $J$ be such that $u_J = N - 1$. Then we have
        $W_{J - 1} W_{J} = ``\mathpzc{ab}" \text{ or } ``\mathpzc{ba}"$ and $W_{J - 1 - k} = W_{J + k}$ for $k = 1, \dots, N - 2$.
    \end{proposition}
    
    \begin{proof}
        Note that $u_{J} \equiv (u_{J - 1} + u_1) \pmod{N}$ because of (Fact 2, Equation \ref{MODD}), and the choice of $J$ gives that $u_{J} = u_{J - 1} + u_1$, and similarly $u_{J + 1} = u_{J} + (u_1 - N)$. This implies $W_{J - 1} \neq W_J$.
        
        We will now inductively prove that for $k = 1, \dots, N - 2$ that $u_{J - k} + u_{J + k} = N - 2$. For $k = 1$ it immediately follows from the formula in the previous paragraph. Now, assume the equality holds at some $1 \leq k < N - 2$. Then, $u_{J - k} - u_1 \geq 0$ if and only if $u_{J + k} + u_1 \leq N - 2$. Therefore, $u_{J - (k + 1)} = u_{J - k} - u_1$ if and only if $u_{J + (k + 1)} = u_{J + k} + u_1$. (Note that it is impossible to have either $u_{J - (k + 1)} = N - 1$ or $u_{J + (k + 1)} = N - 1$ since $k < N - 1$.) If both sides of the if-and-only-if are false, we have  $u_{J - (k + 1)} = u_{J - k} - (u_1 - N)$ and $u_{J + (k + 1)} = u_{J + k} + (u_1 - N)$. In either case, the sum is still preserved. 
        Remark that at each step of the induction, 
            $$u_{J - k} - u_{J - (k + 1)} = u_{J + (k + 1)} - u_{J + k},$$
        thus the gap sizes are the same and hence $W_{J - 1 - k} = W_{J + k}$.
    \end{proof}
    
    Now we describe the symmetry about other gaps in the word at relabeling times.
    
    \begin{proposition} \label{PROP2}
        Let $W$ be the word of $\mathpzc{a}$'s and $\mathpzc{b}$'s at some relabeling time $N~=~R_q$. Then the word satisfies the following symmetry. Let $J$ be such that $u_J = N - p$ for some $1 \leq p \leq u_1 = R_q - R_{q - 1}$. Then we have
        $W_{J - 1} W_{J} = ``\mathpzc{ab}" \text{ or } ``\mathpzc{ba}"$ and $W_{J - 1 - k} = W_{J + k}$ for $k = 1, \dots, \ell$, where $\ell + 1$ is the smallest index such that $\max\{ u_{J - (\ell + 1)}, u_{J + (\ell + 1)}\} \geq u_J$.
    \end{proposition}        
    
    \begin{proof}
        By Facts 1 and 3, since $1 \leq p \leq R_q - R_{q - 1}$ and since removing the point $\{(N - p) \alpha\}$ corresponds to combining an $``\mathpzc{ab}"$ or $``\mathpzc{ba}"$ into a $\mathpzc{c}$, we have that $W_{J - 1} \neq W_{J}$.
        
        % We will now inductively prove that for $k = 1, \dots, \ell$ that $u_{J - k} + u_{J + k} = N - 2p$. For $k = 1$ it immediately follows from the fact that $W_{J - 1} \neq W_{J}$ as in the proof of Proposition \ref{PROP1}. Now, assume the equality holds for some $1 \leq k < \ell$. Then $u_{J - k} - u_1 \geq 0$ if and only if $u_{J + k} + u_1 \leq N - 2p$. Therefore, $u_{J - (k + 1)} = u_{J - k} - u_1$ if and only if $u_{J + (k + 1)} = u_{J + k} + u_1$. (Note that it is impossible to have either $u_{J - (k + 1)} \geq N - 2p$ or $u_{J + (k + 1)} \geq N - 2p$ because of the choice of $\ell$.) Now the proof follows exactly as in the previous proposition.
        
        % We will now inductively prove that for $k = 1, \dots, \ell$ that $u_{J - k} + u_{J + k} = N - 2p$. For $k = 1$ it immediately follows from the fact that $W_{J - 1} \neq W_{J}$ as in the proof of Proposition \ref{PROP1}. Now, assume the equality holds for some $1 \leq k < \ell$. Then $u_{J - k} - u_1 \geq 0$ if and only if $u_{J + k} + u_1 \leq N - 2p$. Therefore, $u_{J - (k + 1)} = u_{J - k} - u_1$ if and only if $u_{J + (k + 1)} = u_{J + k} + u_1$. Note that it is impossible to have either $u_{J + (k + 1)} \geq N - p$ because of the choice of $\ell$, or $N - p >  u_{J + (k + 1)} > N - 2p$ as it implies $u_{J - (k + 1)} > N - p$. Now the proof follows exactly as in the previous proposition.
        
        We will now inductively prove for $k = 1, \dots, \ell$ that 
            \begin{equation}
                u_{J - k} + u_{J + k} = N - 2p. \label{IND_HYP}
            \end{equation}
        For $k = 1$, it immediately follows from the fact that $W_{J - 1} \neq W_{J}$ as in the proof of Proposition \ref{PROP1}. Now, assume the equality holds for some $1 \leq k < \ell$.
        \newpage
        We claim that the following four statements are equivalent:
        \begin{enumerate}
            \item [(1)] $u_{J - (k + 1)} = u_{J - k} - u_1$,
            \item [(2)] $u_{J + (k + 1)} = u_{J + k} + u_1$,
            \item [(3)] $u_{J - k} - u_1 \geq 0$, and
            \item [(4)] $u_{J + k} + u_1 \leq N - 2p$.
        \end{enumerate}
        
        \noindent
        First, note that (1) is equivalent to (3) by Fact 2 (Equation \ref{MODD}). Moreover, (3) is equivalent to (4) by the induction hypothesis (\ref{IND_HYP}). Finally, we show (2) is equivalent to (4). For the forward direction, note that $u_{J + (k + 1)} \geq N - p$ is impossible due to the choice of $\ell$. Now suppose that $N - 2p < u_{J + k} + u_1 < N - p$. By the induction hypothesis (\ref{IND_HYP}) we have $u_{J - k} - u_1 = N - 2p - (u_{J + k} + u_1)$, hence
            $$N - 2p - (N - p) < N - 2p - (u_{J + k} + u_1) < N - 2p - (N - 2p),$$
        or equivalently, $-p < u_{J - k} - u_1 < 0$. This means that $u_{J - (k + 1)} > N - p$, which is again impossible by the choice of $\ell$. Meanwhile, the reverse direction follows immediately from Fact 2 (Equation \ref{MODD}).

    %     It is also true that $u_{J + k} + u_1 \leq N - 2p$ if and only if $u_{J + (k + 1)} = u_{J + k} + u_1$. To see the reverse direction, note that it is impossible to have $u_{J + (k + 1)} \geq N - p$ because of the choice of $l$, and in the other case that $N - 2p < u_{J + k} + u_1 < N - p$, since $u_{J - k} - u_1 = N - 2p - (u_{J + k} + u_1)$, we have
    %         $$N - 2p - (N - p) < N - 2p - (u_{J + k} + u_1) < N - 2p - (N - 2p),$$
    %     or $-p < u_{J - k} - u_1 < 0$, which means $u_{J - (k + 1)} > N - p$, also impossible for the same reason.

        Therefore, $u_{J - (k + 1)} = u_{J - k} - u_1$ if and only if $u_{J + (k + 1)} = u_{J + k} + u_1$, and hence (\ref{IND_HYP}) holds when $k$ is replaced by $k + 1$. Now the proof follows exactly as in the previous proposition.
    \end{proof}
    
    Remark that even though Proposition \ref{PROP1} is a special case of Proposition \ref{PROP2}, we have included both to make the exposition clearer.

    \begin{proof}[Proof of the Symmetry Theorem]
        Let $(R_j)_{j \geq 1}$ be the ordered increasing sequence of relabeling times. It is clear that the word of length $R_j$ satisfies the theorem: there are no $\mathpzc{c}$'s to center the symmetry around. Now, remark that in moving from word $W(R_j)$ to $W(R_j - 1)$, the $\mathpzc{ab}$ or $\mathpzc{ba}$ centered at $\{ (R_j - 1)\alpha \}$ turns into a $\mathpzc{c}$, and the symmetry centered at this $\mathpzc{c}$ must span the entire word, which it indeed does by Proposition \ref{PROP1}.  
        
        Now, consider $W(R_j - k)$ where $1 < k < R_j - R_{j - 1}$. It is obtained from $W(R_j)$ by removing $\{ (N - i)\alpha \}$ for $1 \leq i \leq k$. As each point is removed, either an $\mathpzc{ab}$ or $\mathpzc{ba}$ turns into a $\mathpzc{c}$, and we must prove the asserted symmetry about each $\mathpzc{c}$. However, this corresponds directly to the symmetry proved in Proposition \ref{PROP2}. Note that the condition $\max\{ u_{J - (\ell + 1)}, u_{J + (\ell + 1)}\} \geq u_J$ corresponds to stopping the symmetry at the closest occurring $\mathpzc{c}$ to the left or right of the given one.
    \end{proof}

\noindent
\textbf{Acknowledgments.}
        The second author thanks Pavel Bleher for introducing him to this subject and for many interesting conversations about it. We also thank Val\'erie Berth\'e, Ethan Coven, Alan Haynes, and Ronnie Pavlov for their helpful comments. 
        This work was supported by NSF grant DMS-1348589.

\end{document}